\documentclass[12pt]{article}
\usepackage{amssymb,amsmath,latexsym, amsthm,enumerate,  amsfonts}
\usepackage{wrapfig, tikz}
\usetikzlibrary{matrix, calc}
\usepackage{blindtext}
\title{} \author{} \date{}
\usepackage[pdftex, pdfstartview=FitH]{hyperref} 

\usepackage[paperwidth=210mm, paperheight=297mm,  hmargin={20mm,20mm}, vmargin={20mm,20mm} ]{geometry}

\numberwithin{equation}{section} 

    \usepackage{wrapfig, tikz}
    \usepackage{tikz-cd}

\newcommand{\Cl}{\mathop{\mathrm {Cl}}\nolimits}
\newcommand{\Count}{\mathop{\mathrm {Count}}\nolimits}

\newcommand{\Int}{\mathop{\mathrm {Int}}\nolimits}

 \newtheorem{theorem}{Theorem}
 \newtheorem{cor}{Corollary}

 \theoremstyle{definition}
 
  \newtheorem{example}{Example}
 \theoremstyle{remark}

\begin{document}
\thispagestyle{empty}
\begin{center}
{\large \bf  Relations between several topologies obtained by the ideal of countable sets\footnote{This research was supported by the Science Fund of the Republic of Serbia, Grant No. 7750027: Set-theoretic, model-theoretic and Ramsey-theoretic phenomena in mathematical structures: similarity and diversity–SMART
}

} \vspace*{3mm}

{\bf Aleksandar Pavlovi\'c\footnote{Department of Mathematics and Informatics, Faculty of Sciences, University of @Novi Sad, Serbia, e-mail: \href{mailto:apavlovic@dmi.uns.ac.rs}{apavlovic@dmi.uns.ac.rs}, ORCID iD: \href{https://orcid.org/0000-0001-5001-6869}{orcid.org/0000-0001-5001-6869}}}
\end{center}

\begin{abstract}

The main goal of this paper is to investigate relations between topologies obtained by: $\theta$-open sets,  $\omega$-open sets,  $\theta_\omega$-open sets, local function, and local closure function with ideal of the countable sets. As the result we will obtain some new the relations which hold in general, and, as the ultimate goal, equality of the topology of $\omega$-open sets and topology obtained by the local function with ideal of countable sets.
\\[2mm] {\it AMS Mathematics  Subject Classification $(2010)$}:
54A10, 
54A05,  
54B99, 
54E99 
\\[1mm] {\it Key words and phrases:} ideal topological space, local function, local closure function, $\theta$-open sets,
 $\theta$-closure,  $\omega$-open sets,  $\theta_\omega$-open sets, ideal of countable sets

\end{abstract}

\section{Introduction}

The idea of ``idealizing" a topological space can be found in some classical texts of Kuratowski \cite{KURold, KUR} and Vaidyanathaswamy \cite{Vaid}. Some early applications of ideal topological spaces can be found in various branches of mathematics, like a generalization of Cantor-Bendixson theorem by Freud \cite{Freud}, or  in measure theory by Scheinberg \cite{Sein}. In 1990  Jankovi\'c and Hamlett \cite{JH} wrote a paper in which they, among  their results, included many other results in this area using modern notation, and logically and systematically arranging them. This paper rekindled the interest in this topic, resulting in many generalizations of the ideal topological space and many generalizations of the notion of open sets, like in   papers of Jafari  and  Rajesh \cite{JR}, and Manoharan and   Thangavelu \cite{MT}.

In 1966 Veli\v{c}ko  \cite{Veli} introduced the  notions of $\theta$-open and $\theta$-closed sets, and also a $\theta$-closure, examining $H$-closed spaces in terms of an arbitrary filterbase. A space $X$ is called $H$-closed if every open cover of $X$ has a finite subfamily whose closures cover $X$. It turned out that $\theta$-open sets are related to the already known notion of $\theta$-continuity, introduced  in 1943 by Fomin \cite{Fomin}.  In 1975 Dickman and Porter continued the study of $H$-closed spaces using  $\theta$-closed sets proving that an $H$-closed space is not a countable union of nowhere dense $\theta$-closed sets. Also they proved that  every $H$-closed space with ccc is not a union of less than continuum many $\theta$-closed nowhere dense sets if and only if Martin's axiom holds. In 1982 Long and Herrington applied $\theta$-open sets to describe faintly continuous functions (see \cite{LongHerr}), which are further investigated and compared with other types of continuity by Njamcul and Pavlovi\'{c} (see \cite{NjamPavTurkish}).  In 1980, Jankovi\'c \cite{Jan} proved that a space is Hausdorff if and only if every compact set is $\theta$-closed. Recent  applications of $\theta$-open sets can be found in the paper of Caldas,  Jafari and Latif \cite{CJL}, or in the paper of  Cammaroto, Catalioto,  Pansera and Tsaban \cite{CCPT}.

In \cite{ON}, Al-Omari and Noiri introduced the  local closure function as a generalization of the $\theta$-closure and the   local function in an ideal topological space. They proved some basic properties for the local closure function, and also introduced two new topologies obtained from the original one using the local closure function.

The family of $\omega$-open sets was introduced by Hdeib \cite{Hdeib} in 1982. Within few years from 2013 till 2015, we had several applications of this notion \cite{ Car2, Car3,  Car4, Car5, Car1, Dar, Zor}, mainly in the field of variations of continuous functions.  The family of $\theta_\omega$-open sets was introduced in 2017 by Al Ghour and Irshedat \cite{AlGIrsh}, as a $\theta$ modification of $\omega$-open sets. This research continued in \cite{AlGElIssa}.

Looking at the definitions of topologies defined by local and local closure function with ideal of countable sets, on the one side, and topologies of $\omega$ and $\theta_\omega$ open sets, it is not hard to notice similarities. The goal of this paper is to investigates relations of this topologies.

\section{Definitions and notations}

We will use the following notation. If  $\langle X, \tau\rangle$ is  a topological space,  $\tau(x)$ will be the family of open neighbourhoods at the point $x$, $\Cl(A)$   the closure of the set $A$ and $\Int(A)$ its interior.

  $\omega$   will denote the set of natural numbers including $0$. $\aleph_0$ will be the carinal number of the set of naturals and $\mathfrak{c}$ the cardinal number of the set of reals.

If $X$ is a nonempty set, a family $\mathcal{I}\subset P(X)$ satisfying

{\leftskip 5mm (I0) $\emptyset \in \mathcal{I}$,

(I1) If $A \in  \mathcal{I}$  and $B \subseteq A$, then $B \in \mathcal{I}$,

(I2) If $A, B \in \mathcal{I}$, then $A \cup B \in \mathcal{I}$,

}

\noindent is called an \textbf{ideal} on $X$. If $X \not \in \mathcal{I}$ (i.e.,\ $P(X) \neq \mathcal{I}$), then $\mathcal{I}$ is \textbf{proper}. If there exists $A \subseteq X$ such that $B \in \mathcal{I}$ if $B \subseteq A$, then $\mathcal{I}$ is a \textbf{principal ideal}. The family of countable sets forms an ideal denoted by $\Count$.

\subsubsection*{Local function}

If $\langle X, \tau\rangle$ is a topological space and $\mathcal{I}$ an ideal on $X$, then a triple $\langle X, \tau, \mathcal{I}\rangle$ is called an \textbf{ideal topological space}.
If $\langle X, \tau, \mathcal{I}\rangle$ is an ideal topological space, then the mapping $A \mapsto A^*_{(\tau, \mathcal{I})}$ (briefly $A^*$) defined by
$$A^*_{(\tau, \mathcal{I})}=\{x \in X: A \cap U \not \in \mathcal{I} \mbox { for each }  U \in \tau(x)\}$$
is called the \textbf{local function} (see \cite{KUR}).

The local function has the following properties (see \cite{JH}):

{\leftskip 5mm (1) $A \subseteq B \Rightarrow A^* \subseteq B^*$;

(2) $A^*=\Cl(A^*)\subseteq \Cl(A)$;

(3) $(A^*)^*\subseteq A^*$;

(4) $(A\cup B)^*=A^* \cup B^*$

(5) If $I \in \mathcal{I}$, then $(A\cup I)^*=A^*=(A\setminus I)^*$.

}

Also by $\Cl^*(A)=A \cup A^*$   a closure operator on $P(X)$ is defined and it generates a topology $\tau^*(\mathcal{I})$  (briefly $\tau^*$) on $X$ where $$\tau^*(\mathcal{I})=\{U \subseteq X: \Cl^*(X \setminus U)=X \setminus U\}.$$

It is easy to see that $\tau \subseteq \tau^*\subseteq P(X)$.

\subsubsection*{$\theta$-open sets}

A set $U$ is $\theta$-open if for each $x \in U$ there exists $V \in \tau(x)$ such that $\Cl(V)\subseteq U$. $\theta$-interior is defined by
$\Int_\theta(A)=\bigcup \{U: U \subset A, U\mbox{ is }\theta\mbox{-open}\}$. It is known that $\theta$-interior does not have to be a $\theta$-open set, but $U$ is $\theta$-open if and only if $\Int_\theta(U)=U$.  A set $A$ is $\theta$-closed if $X \setminus A$ is $\theta$-open. Also, a set $A$ is $\theta$-closed if and only if  it is equal to its $\theta$-closure $\Cl_\theta(A)$ defined by
$$\Cl_\theta(A)=\{x \in X: \Cl(U) \cap A \neq \emptyset\mbox{ for each }U \in \tau(x)\}.$$ Similarly as in the case of $\theta$-interior, $\theta$-closure of a set does not have to be $\theta$-closed, but it is always a closed set.  The family of $\theta$-open sets generates a topology $\tau_\theta$ on $X$ and we have that $\tau_\theta \subseteq \tau$. If $\langle X, \tau\rangle$ is a regular space, then every open set is $\theta$-open and therefore $\tau=\tau_\theta$.

\vskip 3mm

\subsubsection*{Local closure function}

Generalizing the notions of local function and $\theta$-closure (see \cite{LongHerr})   Al-Omari and   Noiri \cite{ON}  defined the \textbf{local closure function} in an ideal topological space $\langle X, \tau, \mathcal{I}\rangle$ which assigns to each set $A$ the set  $\Gamma_{(\tau,\mathcal{I})}(A)$ (briefly $\Gamma(A)$), where
$$\Gamma_{(\tau,\mathcal{I})}(A)=\{x \in X: \Cl(U) \cap A \not \in \mathcal {I}\mbox{ for each }U \in \tau(x)\}.$$

The local closure function has the  following properties (see \cite{ON}):

{\leftskip 5mm (1) $A^*\subseteq  \Gamma(A)$;

(2) $\Gamma(A)=\Cl(\Gamma(A))\subseteq \Cl_\theta(A)$;

(3) $\Gamma(A \cup B)=\Gamma(A) \cup \Gamma(B)$;

(4) $\Gamma(A \cup I)=\Gamma(A) =\Gamma(A \setminus I)$ for each $I \in \mathcal{I}$.

}

\noindent Let us notice that   between $\Gamma(\Gamma(A))$ and $\Gamma(A)$  both inclusions are possible.

Using $\Gamma$, Al-Omari and Noiri  \cite{ON} defined an operator $\psi_\Gamma$ by
$\psi_\Gamma(A)=X \setminus \Gamma(X \setminus A)$ which satisfies the following:

{\leftskip 5mm (1) $\psi_\Gamma(A)=\Int(\psi_\Gamma(A))$;

(2) $\psi_\Gamma(A \cap B)= \psi_\Gamma(A) \cap \psi_\Gamma(B)$;

(3) $\psi_\Gamma(A \cup I)=\psi_\Gamma(A)=\psi_\Gamma(A \setminus I)$ for each $I \in \mathcal{I}$;

(4) If $U$ is $\theta$-open, then $U\subseteq\psi_\Gamma(U)$.

}

\noindent Using this,  a topology $\sigma$ is defined in \cite{ON} in the following way:
$$A \in \sigma \Leftrightarrow A\subseteq \psi_\Gamma(A).$$
Also there holds
$$A \in \sigma \Leftrightarrow \Gamma(X \setminus A)\subseteq X \setminus A.$$

Obviously $\tau_\theta \subseteq \sigma$.

Again  in \cite{ON}, using $\psi_\Gamma$, is defined   another topology $\sigma_0$   by
$$A \in \sigma_0 \Leftrightarrow A\subseteq \Int(\Cl(\psi_\Gamma(A))).$$

Since $\psi_\Gamma(A)$ is open,  $\psi_\Gamma(A)\subseteq \Int(\Cl(\psi_\Gamma(A)))$  and we have
$$\tau_\theta \subseteq \sigma \subseteq \sigma_0.$$

More about local closure function can be found in \cite{pav} and \cite{NjamPavPeriodica}.

\subsubsection*{$\omega$-open sets}

In topological space $\langle X, \tau \rangle$ set $A$ which contains all of its $\omega$-accumulation points is called an $\omega$-closed set (see  \cite{Hdeib}). Its complement is called an $\omega$-open set. The family of all $\omega$-open sets is a topology on $X$ denoted by $\tau_\omega$. It is known that $\tau \subseteq \tau_\omega$. It is proved in the same paper that
$$A\in \tau_\omega \Leftrightarrow \forall x \in A \; \exists U \in \tau(x) \; |U \setminus A| \leq \aleph_0.$$

The closure operator in $\tau_\omega$ is denoted by $\overline{A}^\omega$. For more on $\omega$-open sets see \cite{al1999certain}.

\subsubsection*{$\theta_\omega$-open sets}

In 2017 Al Ghour and Irshedat in \cite{AlGIrsh} introduced $\theta_\omega$-open sets as a $\theta$ modification of $\omega$-open sets. They defined $\theta_\omega$-closure, denoted by $\Cl_{\theta_\omega}$ by
$$x \in \Cl_{\theta_\omega}(A)\Leftrightarrow \forall U \in \tau(x)\: \overline{U}^{\omega} \cap A \neq \emptyset.$$
Set $A$ is $\theta_\omega$-closed iff $\Cl_{\theta_\omega}(A)=A$, and $\theta_\omega$-open if its complement is $\theta_\omega$-closed. They proved that $\theta_\omega$-open sets forms topology, and denoted it by $\tau_{\theta_\omega}$.

They proved  $\overline{A}\subseteq \Cl_{\theta_\omega}(A) \subseteq \Cl_{\theta}(A)$, that each $\theta$-closed set is ${\theta_\omega}$-closed and that ${\theta_\omega}$-closed set is closed in $\tau$. They also proved  that every countable open in $\tau$ set is also open in $\tau_{\theta_\omega}$.

 For more details on $\theta_\omega$-open sets check \cite{AlGElIssa}.

\section{Currently known relations between topologies}

It is easy to notice that definitions of topologies $\tau_\omega$ and $\tau_{\theta_\omega}$  are very similar to the definitions of $\tau^*$ and $\sigma$ in the  case when we take for the ideal the ideal of mostly countable set. The goal of this paper is to find relations between topologies $\tau_\omega$ and $\tau_{\theta_\omega}$, on the one side, and $\tau^*(\Count)$, $\sigma(\Count)$ and $\sigma_0(\Count)$, on the other side. Since, we will only deal with the ideal of the most countable sets, the later topologies will be briefly denoted just by $\tau^*$, $\sigma$ and $\sigma_0$, respectably.

Considering the graph of currently known relations, where arrow from  $\tau_1$ to $\tau_2$ mans "$\tau_1 \subseteq \tau_2$ and there exists a topology where strict inclusion holds", the left-hand-side of the graph follows from the comments in the subsection considering $\tau_{\theta_\omega}$. The right-hand-side  of the graph is proved \cite{ON} and the possible  strict inclusion between $\sigma$ and $\sigma_0$ in \cite{pav}. The only arrow between the two side is the consequence of the transitivity.

Besides this, it is known that there are no arrows between $\sigma_0$ and $\tau$.

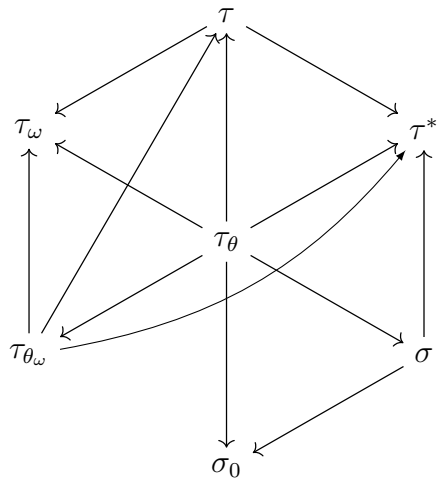
\begin{center}
\begin{tikzpicture}[commutative diagrams/every diagram, scale=1.50]
\node(P0)at (0,0){$\tau_{\theta}$};
\node(P1)at (90:2cm){{$\tau$}} ;
\node(P2)at (90+60:2cm){${\tau_\omega}$};
\node(P3)at (90+120:2cm){$\tau_{\theta_\omega}$};
\node(P4)at (90+180:2cm){$\sigma_0$};
\node(P5)at (90+240:2cm){{$\sigma$}};
\node(P6)at (90+300:2cm){$\tau^*$};

\path[commutative diagrams/.cd,every arrow,every label]
(P0)edgenode { } (P1)
(P0)edgenode { } (P2)
(P0)edgenode { } (P3)
(P0)edgenode { } (P4)
(P0)edgenode { } (P5)
(P0)edgenode { } (P6)
(P1)edgenode { } (P2)
(P3)edgenode { } (P2)
(P3)edgenode { } (P1)
(P5)edgenode { } (P4)
(P5)edgenode { } (P6)
(P1)edgenode { } (P6)
;
\draw[-latex,bend right=20]  (P3) edge (P6);
\end{tikzpicture}
\end{center}

\section{New connections}

In the following sections we will find out few more arrows between the left and the right-hand-side of the graph.

\begin{example}\label{EXtauomegasigma}[$\tau_{\theta_\omega}$ and $\sigma$]
In this example we will prove that these topologies are not comparable in general.

Considering properties of these topologies, we have
\begin{eqnarray*}
A \in \sigma & \Leftrightarrow  & \Gamma(X \setminus A) \subset X \setminus A\\
 & \Leftrightarrow  & \{x \in X : \forall U \in \tau (x) \; |\overline{U}^\tau \cap (X \setminus A)|>\aleph_0\} \subseteq X \setminus A  \\
 & \Leftrightarrow  & \{x \in X : \forall U \in \tau (x) \; |\overline{U}^\tau \setminus A|>\aleph_0\} \subseteq X \setminus A  \\
 & \Leftrightarrow  & A \subseteq X \setminus \{x \in X : \forall U \in \tau (x) \; |\overline{U}^\tau \setminus A|>\aleph_0\}  \\
 & \Leftrightarrow  & \forall x \in A \; \exists U \in \tau (x) \; |\overline{U}^\tau \setminus A|\leq \aleph_0.
\end{eqnarray*}
\begin{eqnarray*}
A \in \tau_{\theta_\omega} & \Leftrightarrow  & X \setminus A =\Cl_{\theta_\omega}(X \setminus A)\\
 & \Leftrightarrow  & X \setminus A =\{x \in X: \forall V \in \tau(x)\; \overline{V}^{\omega}\cap (X \setminus   A)\neq \emptyset \}\\
 & \Leftrightarrow  & X \setminus A =\{x \in X: \forall V \in \tau(x)\; \overline{V}^{\omega}\setminus   A\neq \emptyset \}\\
 & \Leftrightarrow  &  A =\{x \in X: \exists V \in \tau(x)\; \overline{V}^{\omega}\setminus   A  = \emptyset \}\\
 & \Leftrightarrow  &  A =\{x \in X: \exists V \in \tau(x)\; \overline{V}^{\omega}\setminus   A  = \emptyset \}\\
 & \Leftrightarrow  &  \forall x \in A \; \exists V \in \tau(x)\; \overline{V}^{\omega}\setminus   A  = \emptyset \}
\end{eqnarray*}

The last result is proved in  \cite[Theorem 2.23]{AlGIrsh}.

Let $X=\{a\} \cup \mathbb{R}$, where $a \not \in \mathbb{R}$.  The topology  $\tau$ on $X$ is given by a base
$$\mathcal{B} =\left\{ \{a\} \right\} \cup \left\{ \{a,r\}: r \in  \mathbb{R} \right\}.$$

We have $\{a\}\in \tau_{\theta_\omega}$, since  it is a countable open set in $\tau$. On the other hand, for each set $L$ containing $a$ we have $r \in \overline{L}^\tau$, for each $ x \in \mathbb{R}$, implying $\overline{L}^\tau=\mathbb{R}$, so $\forall U \in \tau (a) \; |\overline{U}^\tau \setminus A|= \mathfrak{c}>\aleph_0$.

If we take $\mathbb{R}$, since $|X \setminus  \mathbb{R}|=1 \leq \aleph_0$, we have $\mathbb{R}\in \sigma$. On the other hand, each neighbourhood of a real contains $a$, implying that the closure in any topology of that neighbourhood contains $a$. So, $\overline{V}^{\omega}\setminus   \mathbb{R}=\{a\} \neq \emptyset$.

\end{example}

\begin{theorem}
  $\sigma \subseteq \tau_{\omega}  $
\end{theorem}
\begin{proof}

In the previous example we have proved that $$A \in \sigma\Leftrightarrow\forall x \in A \; \exists U \in \tau (x) \; |\overline{U}^\tau \setminus A|\leq \aleph_0.$$ On the other hand, directly from the definition of $\tau_\omega$, we have
$$A \in \tau_\omega \Leftrightarrow\forall x \in A \; \exists U \in \tau (x) \; |U \setminus A|\leq \aleph_0.$$
From this directly follows $\sigma \subseteq \tau_\omega$.
\end{proof}
\begin{example}
Let us consider the topology defined in  Example \ref{EXtauomegasigma}. We proved that $\{a\} \not \in \sigma$, but, since $\{a\}$ is in $\tau$, it is also in $\tau_\omega$.
\end{example}

\begin{theorem}
  $\tau_\omega \subseteq \tau^*$
\end{theorem}
\begin{proof}
Suppose that $A \in \tau_\omega$. Then
\begin{equation}\label{tauomega}\forall x \in A \; \exists U \in \tau(x) \; |U \setminus A| \leq \aleph_0.\end{equation}
Let us prove that $A \in \tau^*$, i.e.\ $X \setminus A = \Cl^*(X \setminus A)$.

We have
$$X \setminus A = \Cl^*(X \setminus A)\Leftrightarrow X \setminus A = (X \setminus A) \cup  (X \setminus A)^* \Leftrightarrow (X \setminus A)^* \subseteq X \setminus A.$$

\begin{eqnarray*}
x \in (X \setminus A)^* & \Leftrightarrow & \forall U \in \tau(x) \;|(X \setminus A)\cap U| > \aleph_0\\
& \Leftrightarrow & \forall U \in \tau(x)\; |U \setminus A| > \aleph_0.
\end{eqnarray*}
From  \eqref{tauomega} follows that $x \not \in A$, i.e.\ $x \in X \setminus A$, which gives  $X \setminus A = \Cl^*(X \setminus A)$, i.e.\ $A \in \tau^*$.
\end{proof}

\begin{theorem}
  $\tau^* \subseteq \tau_\omega $
\end{theorem}
\begin{proof}By the same conclusion as in the previous theorem, we have that
$$A \in \tau^* \Leftrightarrow \forall x \in X \; \left ( \forall U \in \tau(x)\; |U \setminus A| > \aleph_0 \right ) \Rightarrow x \not \in A.$$
So, if $x \in A$, then $\exists U \in \tau(x) \; |U \setminus A| \leq \aleph_0$, implying $A \in \tau_\omega$.
\end{proof}

\begin{cor}
  $\tau_\omega = \tau^*$
\end{cor}

\section{The final graph}

Considering relations between topologies, we have the following graph.

Besides the results obtained in the previos section, we know that there are no arrows between $\tau_{\theta_\omega}$ and $\sigma$. Also due transitivity of inclusion, there is no arrow from $\sigma_0$ to $\tau_{\theta_\omega}$. For the other direction, we do not know. Considering relations betwee $\sigma$ and $\tau$, Al-Omari and Noiri \cite{ON} showed that those two are not comparable. Also, they proved that $\tau$ is not included in $\sigma_0$.

So, as open questions, it remains to check if $\sigma_0 \subseteq \tau$,   $\tau_{\theta_\omega} \subseteq \sigma_0$ and relations between $\sigma_0$ and $\tau^*$, to complete the graph of relations.

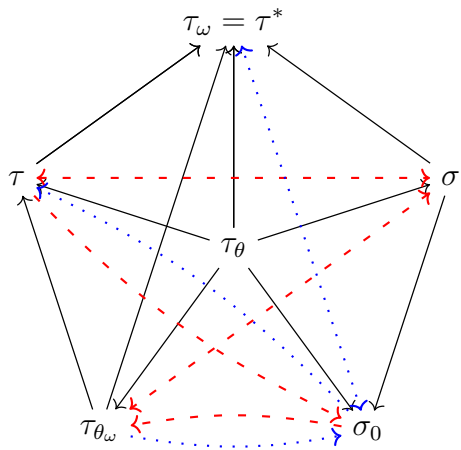
\begin{center}
\begin{tikzpicture}[commutative diagrams/every diagram, scale=1.50]
\node(P0)at (0,0){$\tau_{\theta}$};
\node(P1)at (90+72:2cm){{$\tau$}} ;
\node(P2)at (90+0:2cm){${\tau_\omega}=\tau^*$};
\node(P3)at (90+144:2cm){$\tau_{\theta_\omega}$};
\node(P4)at (90+216:2cm){$\sigma_0$};
\node(P5)at (90+288:2cm){{$\sigma$}};
\node(P6)at (90+270:2cm){};

\path[commutative diagrams/.cd,every arrow,every label]
(P0)edge node { } (P1)
(P0)edge node { } (P2)
(P0)edge node { } (P3)
(P0)edge node { } (P4)
(P0)edge node { } (P5)
(P0)edge node { } (P2)
(P1)edge node { } (P2)
(P3)edge node { } (P2)
(P3)edge node { } (P1)
(P5)edge node { } (P4)
(P5)edge node { } (P2)
(P1)edge node { } (P2)

;

\draw[thick,loosely dashed,red,<->]  (P5) edge (P3);
\draw[thick,loosely dashed,red,<->]  (P1) edge (P5);
\draw[thick,loosely dotted, blue, bend right=10, ->] (P4) edge (P1);
\draw[thick,loosely dotted, blue, bend right=10, ->] (P3) edge (P4);
\draw[thick,loosely dashed, red, bend right=10, ->] (P4) edge (P3);
\draw[thick,loosely dashed, red, bend right=10, ->] (P1) edge (P4);
\draw[thick,loosely dotted,blue,<->]  (P4) edge (P2);
\end{tikzpicture}
\end{center}

By black full arrows are represented relations $\subsetneq$ (in general). By red dashed arrows are represented relations $\not \subset$ (in general), and by dotted blue arrows are represented relations which still have ot be investigated.


\begin{thebibliography}{10}

\bibitem{al1999certain}
S.~Al~Ghour.
\newblock {\em Certain covering properties related to paracompactness}.
\newblock PhD thesis, 1999.

\bibitem{AlGElIssa}
Samer Al~Ghour and Salma El-Issa.
\newblock Theta omega topological operators and some product theorems.
\newblock {\em Mathematical Problems in Engineering}, 2021(1):6438053, 2021.
\newblock URL:
  \url{https://onlinelibrary.wiley.com/doi/abs/10.1155/2021/6438053}, \href
  {https://arxiv.org/abs/https://onlinelibrary.wiley.com/doi/pdf/10.1155/2021/6438053}
  {\path{arXiv:https://onlinelibrary.wiley.com/doi/pdf/10.1155/2021/6438053}},
  \href {https://doi.org/10.1155/2021/6438053}
  {\path{doi:10.1155/2021/6438053}}.

\bibitem{AlGIrsh}
Samer Al~Ghour and Bayan Irshedat.
\newblock The topology of {$\theta_\omega$}-open sets.
\newblock {\em Filomat}, 31(16):5369--5377, 2017.
\newblock \href {https://doi.org/10.2298/fil1716369a}
  {\path{doi:10.2298/fil1716369a}}.

\bibitem{ON}
Ahmad Al-Omari and Takashi Noiri.
\newblock Local closure functions in ideal topological spaces.
\newblock {\em Novi Sad J. Math.}, 43(2):139--149, 2013.

\bibitem{CJL}
M.~Caldas, S.~Jafari, and R.~M. Latif.
\newblock Sobriety via {$\theta$}-open sets.
\newblock {\em An. \c{S}tiin\c{t}. Univ. Al. I. Cuza Ia\c{s}i. Mat. (N.S.)},
  56(1):163--167, 2010.
\newblock \href {https://doi.org/10.2478/v10157-010-0011-y}
  {\path{doi:10.2478/v10157-010-0011-y}}.

\bibitem{CCPT}
Filippo Cammaroto, Andrei Catalioto, Bruno~Antonio Pansera, and Boaz Tsaban.
\newblock On the cardinality of the {$\theta$}-closed hull of sets.
\newblock {\em Topology Appl.}, 160(18):2371--2378, 2013.
\newblock \href {https://doi.org/10.1016/j.topol.2013.07.031}
  {\path{doi:10.1016/j.topol.2013.07.031}}.

\bibitem{Car2}
C.~Carpintero, N.~Rajesh, E.~Rosas, and S.~Saranyasri.
\newblock Some properties of upper/lower {$\omega$}-continuous multifunctions.
\newblock {\em Sci. Stud. Res. Ser. Math. Inform.}, 23(2):35--55, 2013.

\bibitem{Car3}
C.~Carpintero, N.~Rajesh, E.~Rosas, and S.~Saranyasri.
\newblock On upper and lower almost contra-{$\omega$}-continuous
  multifunctions.
\newblock {\em Ital. J. Pure Appl. Math.}, (32):445--460, 2014.

\bibitem{Car4}
C.~Carpintero, N.~Rajesh, E.~Rosas, and S.~Saranyasri.
\newblock Somewhat {$\omega$}-continuous functions.
\newblock {\em Sarajevo J. Math.}, 11(23)(1):131--137, 2015.
\newblock \href {https://doi.org/10.5644/SJM.11.1.11}
  {\path{doi:10.5644/SJM.11.1.11}}.

\bibitem{Car5}
C.~Carpintero, N.~Rajesh, E.~Rosas, and S.~Saranyasri.
\newblock Upper and lower {$(\delta,\omega)$}-continuous multifunctions.
\newblock {\em Afr. Mat.}, 26(3-4):399--405, 2015.
\newblock \href {https://doi.org/10.1007/s13370-013-0211-1}
  {\path{doi:10.1007/s13370-013-0211-1}}.

\bibitem{Car1}
Carlos Carpintero, Namelgahesj Rajesh, Ennis Rosas, and Sanya Saranyasri.
\newblock Properties of faintly {$\omega$}-continuous functions.
\newblock {\em Bol. Mat.}, 20(2):135--143, 2013.

\bibitem{Dar}
Halgwrd~M. Darwesh.
\newblock Between preopen and open sets in topological spaces.
\newblock {\em Thai J. Math.}, 11(1):143--155, 2013.

\bibitem{Fomin}
S.~Fomin.
\newblock Extensions of topological spaces.
\newblock {\em Ann. of Math. (2)}, 44:471--480, 1943.
\newblock \href {https://doi.org/10.2307/1968976} {\path{doi:10.2307/1968976}}.

\bibitem{Freud}
G.~Freud.
\newblock Ein {B}eitrag zu dem {S}atze von {C}antor und {B}endixson.
\newblock {\em Acta Math. Acad. Sci. Hungar.}, 9:333--336, 1958.
\newblock \href {https://doi.org/10.1007/BF02020262}
  {\path{doi:10.1007/BF02020262}}.

\bibitem{Hdeib}
H.~Z. Hdeib.
\newblock {$\omega $}-closed mappings.
\newblock {\em Rev. Colombiana Mat.}, 16(1-2):65--78, 1982.

\bibitem{JR}
S.~Jafari and N.~Rajesh.
\newblock Generalized closed sets with respect to an ideal.
\newblock {\em Eur. J. Pure Appl. Math.}, 4(2):147--151, 2011.

\bibitem{Jan}
Dragan Jankovi\'{c}.
\newblock On some separation axioms and {$\theta $}-closure.
\newblock {\em Mat. Vesnik}, 4(17)(32)(4):439--449, 1980.

\bibitem{JH}
Dragan Jankovi\'{c} and T.~R. Hamlett.
\newblock New topologies from old via ideals.
\newblock {\em Amer. Math. Monthly}, 97(4):295--310, 1990.
\newblock \href {https://doi.org/10.2307/2324512} {\path{doi:10.2307/2324512}}.

\bibitem{KURold}
K.~Kuratowski.
\newblock {\em Topologie I}.
\newblock Warszawa, 1933.

\bibitem{KUR}
K.~Kuratowski.
\newblock {\em Topology. {V}ol. {I}}.
\newblock New edition, revised and augmented. Translated from the French by J.
  Jaworowski. Academic Press, New York-London; Pa\'{n}stwowe Wydawnictwo
  Naukowe, Warsaw, 1966.

\bibitem{LongHerr}
Paul~E. Long and Larry~L. Herrington.
\newblock The {$T_{\theta }$}-topology and faintly continuous functions.
\newblock {\em Kyungpook Math. J.}, 22(1):7--14, 1982.

\bibitem{MT}
R.~Manoharan and P.~Thangavelu.
\newblock Some new sets and topologies in ideal topological spaces.
\newblock {\em Chinese Journal of Mathematics}, 2013:Article ID 973608, 2013.

\bibitem{NjamPavPeriodica}
Anika Njamcul and Aleksandar Pavlovi\'c.
\newblock On closure compatibility of ideal topological spaces and idempotency
  of the local closure function.
\newblock {\em Period. Math. Hungar.}, 84(2):221--234, 2022.
\newblock \href {https://doi.org/10.1007/s10998-021-00401-1}
  {\path{doi:10.1007/s10998-021-00401-1}}.

\bibitem{NjamPavTurkish}
Anika Njamcul and Aleksandar Pavlovi\'c.
\newblock Various types of continuity and their interpretations in ideal
  topological spaces.
\newblock {\em Turkish J. Math.}, 47(7):2086--2097, 2023.

\bibitem{pav}
Aleksandar Pavlovi\'{c}.
\newblock Local function versus local closure function in ideal topological
  spaces.
\newblock {\em Filomat}, 30(14):3725--3731, 2016.
\newblock \href {https://doi.org/10.2298/FIL1614725P}
  {\path{doi:10.2298/FIL1614725P}}.

\bibitem{Sein}
Stephen Scheinberg.
\newblock Topologies which generate a complete measure algebra.
\newblock {\em Advances in Math.}, 7:231--239 (1971), 1971.
\newblock \href {https://doi.org/10.1016/S0001-8708(71)80004-X}
  {\path{doi:10.1016/S0001-8708(71)80004-X}}.

\bibitem{Vaid}
R.~Vaidyanathaswamy.
\newblock The localisation theory in set-topology.
\newblock {\em Proc. Indian Acad. Sci., Sect. A.}, 20:51--61, 1944.

\bibitem{Veli}
N.~V. Velichko.
\newblock The localisation theory in set-topology.
\newblock {\em Mat. Sb. (N.S.)}, 70(112)(1):98--112, 1966.

\bibitem{Zor}
\.Idris Zorlutuna.
\newblock {$\omega$}-continuous multifunctions.
\newblock {\em Filomat}, 27(1):165--172, 2013.
\newblock \href {https://doi.org/10.2298/FIL1301165Z}
  {\path{doi:10.2298/FIL1301165Z}}.

\end{thebibliography}
\end{document}